\documentclass[a4paper,twoside,12pt]{article}

\usepackage{amssymb,amsmath,amsthm,latexsym}
\usepackage{amsfonts}
\usepackage{graphicx,caption,subcaption}
\usepackage[pdftex,bookmarks]{hyperref}
\usepackage[hmargin=1.2in,vmargin=1.2in]{geometry}
\usepackage[mathscr]{euscript}

\usepackage[auth-lg]{authblk}
\allowdisplaybreaks

\usepackage{float}
\usepackage{calc,tikz}
\usetikzlibrary{arrows,decorations.markings}
\tikzstyle{vertex}=[circle, draw, inner sep=1pt, minimum size=8pt]

\usepackage{multirow}
\usepackage{hhline}
\usepackage{booktabs}

\newcommand{\noi}{\noindent}

\newcommand{\cC}{\mathcal{C}}

\newcommand{\sS}{\mathscr{S}_{\chi}}
\newcommand{\cS}{\mathscr{S}_{\chi^-}}
\newcommand{\dS}{\mathscr{S}_{\chi^+}}
\newtheorem{theorem}{Theorem}[section]
\newtheorem{definition}[theorem]{Definition}

\newtheorem{proposition}[theorem]{Proposition}

\title{\textbf{\sc Chromatic Schultz Polynomial of Certain Graphs}}

\author{\sc Sudev Naduvath}
\affil{\small Department of Mathematics\\ CHRIST (Deemed to be University) \\Bangalore-560029, INDIA.\\{\tt sudev.nk@christuniversity.in}}

\date{}

\begin{document}
\maketitle

\begin{abstract}
A topological index of a graph $G$ is a real number which is preserved under isomorphism. Extensive studies on certain polynomials related to these topological indices have also been done recently. In a similar way, chromatic versions of certain topological indices and the related polynomials have also been discussed in the recent literature. In this paper, the chromatic version of the Schultz polynomial is introduced and determined this polynomial for certain fundamental graph classes.
\end{abstract}

\noi\textbf{Keywords:} Graph colouring, Schultz polynomial, $\chi^-$-chromatic Schultz polynomial, $\chi^+$-chromatic Schultz polynomial, modified chromatic Schultz polynomial.

\vspace{0.25cm}

\noi \textbf{Mathematics Subject Classification 2010:} 05C15, 05C31. 


\section{Introduction}

For all  terms and definitions, not defined specifically in this paper, we refer to \cite{BM1,BLS,FH,DBW}. Further, for graph colouring, see \cite{CZ1,JT1,MK1}. Unless mentioned otherwise, all graphs considered here are undirected, simple, finite and connected.

A \textit{proper vertex colouring} of a graph $G$ is an assignment $\varphi:V(G)\to \cC$ of the vertices of $G$, where $\cC=\{c_1,c_2,c_3,\ldots,c_{\ell}\}$ is a set of colours such that adjacent vertices of $G$ have different colours. The cardinality of the minimum set of colours which allows a proper colouring of $G$ is called the \textit{chromatic number} of $G$ and is denoted $\chi(G)$. The set of all vertices of $G$ which have the colour $c_i$ is called the \textit{colour class} of that colour $c_i$ in $G$. The cardinality of the colour class of a colour $c_i$ is said to be the \textit{strength} of that colour in $G$ and is denoted by $\theta(c_i)$. We can also define a function $\zeta:V(G)\to \{1,2,3,\ldots,\ell\}$ such that $\zeta(v_i)=s$ if and only if $\varphi(v_i)=c_s,c_s\in \cC$. 

A vertex colouring consisting of the colours having minimum subscripts may be called a \textit{minimum parameter colouring} (see \cite{KSM}). If we colour the vertices of $G$ in such a way that $c_1$ is assigned to maximum possible number of vertices, then $c_2$ is assigned to maximum possible number of remaining uncoloured vertices and proceed in this manner until all vertices are coloured, then such a colouring is called a \textit{$\chi^-$-colouring} of $G$. In a similar manner, if $c_\ell$ is assigned to maximum possible number of vertices, then $c_{\ell-1}$ is assigned to maximum possible number of remaining uncoloured vertices and proceed in this manner until all vertices are coloured, then such a colouring is called a \textit{$\chi^+$-colouring} of $G$.

A \textit{topological index} of a graph $G$ is a real number which is preserved under isomorphism. The chromatic versions of certain topological indices have been introduced in \cite{KSM}. In this paper, we discuss the chromatic versions of certain polynomials related to te topological indices of a graph $G$.


\section{Chromatic Schultz Polynomial of Graphs}

Note that throughout this study, we use the chromatic colourings of the graphs under consideration. Motivated by the studies on Schultz polynomial of graphs (see \cite{ET1,EYY}), we can now introduce the chromatic version of the Schultz polynomial as follows:

\begin{definition}{\rm 
Let $G$ be a connected graph with chromatic number $\chi(G)$. Then, the \textit{chromatic Schultz polynomial} of $G$, denoted by $\sS(G,x)$, is defined as $$\sS(G,x)=\sum\limits_{u,v\in V(G)}(\zeta(u)+\zeta(v))x^{d(u,v)}.$$  
}\end{definition}

\begin{definition}{\rm 
Let $G$ be a connected graph with chromatic number $\varphi^-$ and $varphi^+$ be the minimal and maximal parameter colouring of $G$. Then, 
\begin{enumerate}\itemsep0mm
\item[(i)] the \textit{$\chi^-$-chromatic Schultz polynomial} of $G$, denoted by $\cS(G,x)$, is defined as $$\cS(G,x)=\sum\limits_{u,v\in V(G)}(\zeta_{\varphi^-}(u)+\zeta_{\varphi^-}(v))x^{d(u,v)};$$ and 
\item[(ii)] the \textit{$\chi^+$-chromatic Schultz polynomial} of $G$, denoted by $\dS(G,x)$, is defined as $$\dS(G,x)=\sum\limits_{u,v\in V(G)}(\zeta_{\varphi^+}(u)+\zeta_{\varphi^+}(v))x^{d(u,v)}.$$
\end{enumerate} 
}\end{definition}

Now, we can determine the chromatic Schultz polynomials of certain fundamental graph classes.

\subsection{Chromatic Schultz Polynomials of Paths}

\noi In this section, we discuss the two types of Schultz polynomials of paths. 

\begin{theorem}\label{Thm-Pn1}
Let $P_n$ be a path on $n$ vertices. Then, we have 
$$\cS(P_n,x)=
\begin{cases}
\sum\limits_{i=0}^{\frac{n-1}{2}}\left[(3n-6i-3)x+(3n-6i-1)\right]x^{2i}; & \text{if $n$ is odd};\\
3\cdot \sum\limits_{i=0}^{n}(n-i)x^i; & \text{if $n$ is even}.
\end{cases}$$
\end{theorem}
\begin{proof}
Let $V=\{v_1,v_2,\ldots,v_n\}$ be the vertex set of $P_n$, where the vertices are labelled consecutively from one end vertex to the other. Note that $\chi(P_n)=2$. Let $c_1,c_2$ be the two colours we use for colouring $P_n$. We also note that te diameter of $P_n$ is $n-1$. Hence, the power of the variable $x$ varies from $0$ to $n-1$ in the Schultz polynomial of $P_n$. Here, we need to consider the following two cases:

\textit{Case-1:} Let $n$ be odd. Then, with respect to a $\chi^-$-colouring, the vertices $v_1,v_3, v_5,\ldots v_n$ get the colour $c_1$ and the vertices $v_2,v_4, v_6,\ldots v_{n-1}$ get the colour $c_2$. The possible colour pairs and their numbers in $G$ in terms of the distances between them are listed in the following table.

\begin{table}[ph]
\centering
\begin{tabular}{|c|c|c|c|}
\hline
Distance $d(u,v)$ & Colour pairs & Number of pairs & Total number of pairs\\ 
\hline
\multirow{2}{*}{0} & $(c_1,c_1)$ & $\frac{n+1}{2}$ & \multirow{2}{*}{$n$}\\
\cline{2-3}  & $(c_2,c_2)$ & $\frac{n-1}{2}$ & \\
\hline
$1$ & $(c_1,c_2)$ & $n-1$ & $n-1$ \\
\hline
\multirow{2}{*}{$2$} & $(c_1,c_1)$ & $\frac{n-1}{2}$ & \multirow{2}{*}{$n-2$}\\
\cline{2-3}  & $(c_2,c_2)$ & $\frac{n-3}{2}$ & \\
\hline
$3$ & $(c_1,c_2)$ & $n-3$ & $n-3$ \\
\hline
\multirow{2}{*}{$4$} & $(c_1,c_1)$ & $\frac{n-3}{2}$ & \multirow{2}{*}{$n-4$}\\
\cline{2-3}  & $(c_2,c_2)$ & $\frac{n-5}{2}$ & \\
\hline
$5$ & $(c_1,c_2)$ & $n-5$ & $n-5$ \\
\hline
\multirow{2}{*}{$6$} & $(c_1,c_1)$ & $\frac{n-5}{2}$ & \multirow{2}{*}{$n-6$}\\
\cline{2-3}  & $(c_2,c_2)$ & $\frac{n-7}{2}$ & \\
\hline
 ........ & ........... & ............. & ...........\\
 ........ & ........... & ............. & ...........\\ 
\hline
\multirow{2}{*}{$n-3$} & $(c_1,c_1)$ & $2$ & \multirow{2}{*}{$3$}\\
\cline{2-3}  & $(c_2,c_2)$ & $1$ & \\
\hline
$n-2$ & $(c_1,c_2)$ & $2$ & $2$ \\
\hline
\multirow{2}{*}{$n-1$} & $(c_1,c_1)$ & $1$ & \multirow{2}{*}{$1$}\\
\cline{2-3}  & $(c_2,c_2)$ & $0$ & \\
\hline
\end{tabular}
\caption{}
\label{Pn-odd}
\end{table}

In the above table, the possible distances between different pairs of vertices are written in the first column, the different colour pairs with respect to each distance is written in the second column and the number of corresponding colour pairs with respect to each distance is written in the third column. The total number of vertex pairs corresponding to each distance is written in the fourth column.

From Table-\ref{Pn-odd}, we note that for $1\le r\le n$, the number of vertex pairs which are at a distance $r$ is $n-r$ and in this case all colour pairs contain two colours. But, when $r$ is even, all colour pairs contain the same colour - either $(c_1,c_1)$ or $(c_2,c_2)$. In this case, note that the number of $(c_1,c_1)$-colour pairs is $\frac{n-r+1}{2}$ and the number of $(c_2,c_2)$-colour pairs is $\frac{n-r-1}{2}$ so that total number of colour pairs is $n-r$.
Hence, \begin{eqnarray*}
\cS(P_n,x) & = & \sum\limits_{r\ {\rm odd}}(1+2)(n-r)x^r+\sum\limits_{r\ {\rm even}}\left[\frac{n-r+1}{2}\cdot 2+\frac{n-r-1}{2}\cdot 4\right]x^r\\
& = & \sum\limits_{r\ {\rm odd}}(3n-3r)x^r+\sum\limits_{r\ {\rm even}}(3n-3r-1)x^r\\
& = & \sum\limits_{i=0}^{\frac{n-1}{2}}(3n-6i-3)x^{2i+1}+ \sum\limits_{i=0}^{\frac{n-1}{2}}(3n-6i-1)x^{2i}\\
& = & \sum\limits_{i=0}^{\frac{n-1}{2}}\left[(3n-6i-3)x+(3n-6i-1)\right]x^{2i}.
\end{eqnarray*}

\textit{Case-2:} Let $n$ be even. Then, Then, with respect to a $\chi^-$-colouring, the vertices $v_1,v_3, v_5,\ldots v_{n-1}$ get the colour $c_1$ and the vertices $v_2,v_4, v_6,\ldots v_{n}$ get the colour $c_2$. The possible colour pairs and their numbers in $G$ in terms of the distances between them are listed in the following table.

\begin{table}[ph]
\centering
\begin{tabular}{|c|c|c|c|}
\hline
Distance $d(u,v)$ & Colour pairs & Number of pairs & Total number of pairs\\
\hline
\multirow{2}{*}{0} & $(c_1,c_1)$ & $\frac{n}{2}$ & \multirow{2}{*}{$n$}\\
\cline{2-3}  & $(c_2,c_2)$ & $\frac{n}{2}$ & \\
\hline
$1$ & $(c_1,c_2)$ & $n-1$ & $n-1$ \\
\hline
\multirow{2}{*}{$2$} & $(c_1,c_1)$ & $\frac{n-2}{2}$ & \multirow{2}{*}{$n-2$}\\
\cline{2-3}  & $(c_2,c_2)$ & $\frac{n-2}{2}$ & \\
\hline
$3$ & $(c_1,c_2)$ & $n-3$ & $n-3$ \\
\hline
\multirow{2}{*}{$4$} & $(c_1,c_1)$ & $\frac{n-4}{2}$ & \multirow{2}{*}{$n-4$}\\
\cline{2-3}  & $(c_2,c_2)$ & $\frac{n-4}{2}$ & \\
\hline
$5$ & $(c_1,c_2)$ & $n-5$ & $n-5$ \\
\hline
\multirow{2}{*}{$6$} & $(c_1,c_1)$ & $\frac{n-6}{2}$ & \multirow{2}{*}{$n-6$}\\
\cline{2-3}  & $(c_2,c_2)$ & $\frac{n-6}{2}$ & \\
\hline
........ & ........... & ............. & ...........\\
........ & ........... & ............. & ...........\\ 
\hline
$n-3$ & $(c_1,c_2)$ & $3$ & $3$ \\
\hline
\multirow{2}{*}{$n-2$} & $(c_1,c_1)$ & $1$ & \multirow{2}{*}{$2$}\\
\cline{2-3}  & $(c_2,c_2)$ & $1$ & \\
\hline
$n-1$ & $(c_1,c_2)$ & $1$ & $1$ \\
\hline
\end{tabular}
\caption{}
\label{Pn-even}
\end{table}

\noi From Table-\ref{Pn-even}, we have
\begin{eqnarray*}
\cS(P_n,x) & = & \sum\limits_{r\ {\rm odd}}(1+2)(n-r)x^r+\sum\limits_{r\ {\rm even}}\left[\frac{n-r}{2}\cdot 2+\frac{n-r}{2}\cdot 4\right]x^r\\
& = & \sum\limits_{r\ {\rm odd}}(3n-3r)x^r+\sum\limits_{r\ {\rm even}}(3n-3r)x^r\\
& = & 3\cdot \sum\limits_{i=0}^{n}(n-i)x^i.
\end{eqnarray*}
\noi This completes the proof.
\end{proof}

Note that the $\chi^+$-colouring of $P_n$ can be obtained by interchanging the colours $c_1$ and $c_2$ in the $\chi^-$-colouring. Hence, as explained in the proof of above theorem, we have 

\begin{theorem}\label{Thm-Pn2}
Let $P_n$ be a path on $n$ vertices. Then, we have 
$$\dS(P_n,x)=
\begin{cases}
\sum\limits_{i=0}^{\frac{n-1}{2}}\left[(3n-6i-3)x+(3n-6i+1)\right]x^{2i}; & \text{if $n$ is odd};\\
3\cdot \sum\limits_{i=0}^{n}(n-i)x^i; & \text{if $n$ is even}.
\end{cases}$$
\end{theorem}

\subsection{Chromatic Schultz Polynomial of Cycles}

\begin{theorem}\label{Thm-Cn1}
Let $C_n$ be a path on $n$ vertices. Then, we have 
$$\cS(C_n,x)=
\begin{cases}
\frac{3n(1-x^{\frac{n+2}{2}})}{1-x}; & \text{if $n$ is even};\\
\frac{3(n+1)(1-x^{\frac{n+3}{2}})}{1-x}; & \text{if $n$ is odd}.
\end{cases}$$
\end{theorem}
\begin{proof}
Let $V=\{v_1,v_2,\ldots,v_n\}$ be the vertex set of $C_n$, where the vertices are labelled consecutively from one end vertex to the other in a clockwise manner. Note that if $n$ is odd, then the diameter of $C_n$ is $\frac{n+1}{2}$ and if $n$ is even, the diameter of $C_n$ is $\frac{n}{2}$. Hence, we have to consider the following two cases:

\textit{Case-1:} Let $n$ be even. Then, $C_n$ is $2$-colourable and we can label the vertices $v_1,v_3,v_5\ldots,v_{n-1}$ by colour $c_1$ and the vertices $v_2,v_4,v_6\ldots,v_{n}$ by colour $c_2$. Then, for $0\le i\le n$, the possible colour pairs and their numbers can be obtained from the following table.

\begin{table}[ph]
\centering
\begin{tabular}{|c|c|c|c|}
\hline
Distance $d(u,v)$ & Colour pairs & Number of pairs & Total number of pairs\\
\hline
\multirow{2}{*}{$i$, even} & $(c_1,c_1)$ & $\frac{n}{2}$ & \multirow{2}{*}{$n$}\\
\cline{2-3}  & $(c_2,c_2)$ & $\frac{n}{2}$ & \\
\hline
$i$, odd & $(c_1,c_2)$ & $n$ & $n$ \\
\hline
\end{tabular}
\caption{}
\label{Cn-even}
\end{table}

Then, from Table \ref{Cn-even}, we have 
\begin{eqnarray*}
\cS(C_n,x) & = & \sum\limits_{i\ \text{odd}}3nx^i+\sum\limits_{i\ \text{even}}\left[\frac{n}{2}\cdot 2+\frac{n}{2}\cdot 4\right]x^i\\
& = & \sum\limits_{i\ \text{odd}}3nx^i+ \sum\limits_{i\ \text{even}}(n+2n)x^i\\
& = & \sum\limits_{i=0}^{\frac{n}{2}}3nx^i\\
& = & \frac{3n(1-x^{\frac{n+2}{2}})}{1-x}.
\end{eqnarray*}   

\textit{Case-2:} Let $n$ be odd. Then, $\chi(C_n)=3$ and the vertices $v_1,v_3,v_5\ldots,v_{n-1}$ by colour $c_1$ and the vertices $v_2,v_4,v_6\ldots,v_{n-2}$ by colour $c_2$ and the vertex $v_n$ gets colour $c_3$. Then, for $0\le i\le n$, the possible colour pairs and their numbers can be obtained from Table \ref{Cn-odd}.
\begin{table}[h]
\centering
\begin{tabular}{|c|c|c|c|}
\hline
Distance $d(u,v)$ & Colour pairs & Number of pairs & Total number of pairs\\
\hline
\multirow{3}{*}{$i=0$} & $(c_1,c_1)$ & $\frac{n-1}{2}$ & \multirow{3}{*}{$n$}\\
\cline{2-3}  & $(c_2,c_2)$ & $\frac{n-1}{2}$ & \\
\cline{2-3}  & $(c_3,c_3)$ & $1$ & \\
\hline
\multirow{5}{*}{$i>0$ and even} & $(c_1,c_1)$ & $\frac{n-r-1}{2}$ & \multirow{5}{*}{$n$}\\
\cline{2-3}  & $(c_1,c_2)$ & $r-1$ & \\
\cline{2-3}  & $(c_2,c_2)$ & $\frac{n-r-1}{2}$ & \\
\cline{2-3}  & $(c_1,c_3)$ & $1$ & \\
\cline{2-3}  & $(c_2,c_3)$ & $1$ & \\
\hline
\multirow{5}{*}{$i>$, odd} & $(c_1,c_1)$ & $\frac{r-1}{2}$ & \multirow{5}{*}{$n$}\\
\cline{2-3}  & $(c_1,c_2)$ & $n-r-1$ & \\
\cline{2-3}  & $(c_2,c_2)$ & $\frac{r-1}{2}$ & \\
\cline{2-3}  & $(c_1,c_3)$ & $1$ & \\
\cline{2-3}  & $(c_2,c_3)$ & $1$ & \\
\hline
\end{tabular}
\caption{}
\label{Cn-odd}
\end{table}

When $i=0$, we have
\begin{eqnarray*}
\sum\limits_{v\in V}(\zeta(v)+\zeta(v))x^{d(v,v)} & = &\left[(2+4)\cdot \frac{n-1}{2}+6\cdot 1\right]x^0\\
& = & 3(n+1)x^0\\
\end{eqnarray*}
When $i>0$ and is even, we have
\begin{eqnarray*}
\sum\limits_{d(u,v)=i}(\zeta(u)+\zeta(v))x^{d(u,v)} & = &\left[(2+4)\cdot \frac{n-r-1}{2}+3(r-1)+(4+5)\cdot 1\right]x^i\\
& = & 3(n+1)x^i\\
\end{eqnarray*}

Similarly, when $i>0$ and is odd, we have
\begin{eqnarray*}
\sum\limits_{d(u,v)=i}(\zeta(u)+\zeta(v))x^{d(u,v)} & = &\left[(2+4)\cdot \frac{r-1}{2}+3(n-r-1)+(4+5)\cdot 1\right]x^i\\
& = & 3(n+1)x^i\\
\end{eqnarray*}
Therefore, $\cS(C_n,x)=\sum\limits_{i=0}^{\frac{n+1}{2}}3(n+1)x^i=\frac{3(n+1)(1-x^{\frac{n+3}{2}})}{1-x}$, completing the proof.
\end{proof}

Note that in the $\chi^-$-colouring of an even cycle $C_n$ if we the colours $c_1$ and $c_2$, we get its $\chi^+$-colouring. It can be observed that this change makes no change in the corresponding Schultz polynomial.  But, for an odd cycle $C_n$, we have to interchange the colours $c_1$ and $c_3$ in its $\chi^-$-colouring and keep $c_2$ as it is to get a $\chi^+$-colouring. 

In view of this fact, the $\chi^+$-chromatic Schultz polynomial of $C_n$ is obtained in the following theorem. 

\begin{theorem}\label{Thm-Cn2}
Let $C_n$ be a path on $n$ vertices. Then, we have 
$$\dS(C_n,x)=
\begin{cases}
\frac{3n(1-x^{\frac{n+2}{2}})}{1-x}; & \text{if $n$ is even};\\
\frac{(5n-3)(1-x^{\frac{n+3}{2}})}{1-x}; & \text{if $n$ is odd}.
\end{cases}$$
\end{theorem}

\subsection{Chromatic Schultz Polynomial of Complete Graphs}

 Next, we consider the complete graph $K_n$. In $K_n$, we have $d(u,v)=1$ for any two $u,v\in V(G)$. Therefore, $\cS(K_n,x)$ and $\dS(K_n,x)$ are the same and are first degree polynomials. The following result provides the Schultz polynomial of a complete graph $K_n$.
 
\begin{proposition}\label{Prop-Kn}
For $n\ge 2$, $\cS(K_n,x)=\dS(K_n,x)=(n^2+n)+(2n^2-n-3)x$.
\end{proposition}
\begin{proof}
In any proper colouring, distinct vertices in $K_n$ get distinct colours. Now, $\sum\limits_{v\in V}2\zeta(v)x^0=(2+4+6+\ldots+2n)x^0=n(n+1)$. Also, we have 
\begin{eqnarray*}
\sum\limits_{d(u,v)=1}(\zeta(u)+\zeta(v))x^1 & = & (3+4+5+\ldots+(2n-1))x=\left(\frac{2n-3}{2}(2n+2)\right)x\\
& = & (2n-3)(n+1)x.
\end{eqnarray*} 
Therefore, $\cS(K_n,x)=(n^2+n)+(2n^2-n-3)x=\dS(K_n,x)$.
\end{proof}

\noi Next, let us consider the complete bipartite graphs $K_{a,b}$, where $a\ge b$.

\begin{theorem}\label{Thm-Kab}
For a complete bipartite $K_{a,b}, a\ge b, a+b=n$, we have  $\cS(K_n,x)=(2a+4b)+3abx+(a(a-1)+2b(b-1))x^2$ and  $\dS(K_n,x)=(4a+2b)+3abx+(2a(a-1)+b(b-1))x^2$.
\end{theorem}
\begin{proof}
Note that $K_{a,b}$ is $2$-colourable and its diameter is $2$. Since $a\ge b$, with respect to all $a$ vertices in the first partition get the colour $c_1$ and all $b$ vertices in the second partition get colour $c_2$. Then, we have the following table.
\begin{table}[ph]
\centering
\begin{tabular}{|c|c|c|c|}
\hline
Distance $d(u,v)$ & Colour pairs & Number of pairs & Total number of pairs\\
\hline
\multirow{2}{*}{$i=0$} & $(c_1,c_1)$ & $a$ & \multirow{2}{*}{$a+b$}\\
\cline{2-3}  & $(c_2,c_2)$ & $b$ & \\
\hline
$i=1$ & $(c_1,c_2)$ & $ab$ & $ab$ \\
\hline
\multirow{2}{*}{$i=2$} & $(c_1,c_1)$ & $\binom{a}{2}$ & \multirow{2}{*}{$\binom{a}{2}+\binom{b}{2}$}\\
\cline{2-3}  & $(c_2,c_2)$ & $\binom{b}{2}$ & \\
\hline
\end{tabular}
\caption{}
\label{Kab}
\end{table}
Then, 
\begin{eqnarray*}
\cS(K_{m,n},x)& = & (2a+4b)+3abx+ (2\cdot \binom{a}{2}+4\cdot \binom{b}{2})x^2\\
& = & (2a+4b)+3abx+(a(a-1)+2b(b-1))x^2.
\end{eqnarray*}

In a similar way, by interchanging $c_1$ and $c_2$, we can prove that $\dS(K_{m,n},x)=(4a+2b)+3abx+(2a(a-1)+b(b-1))x^2$.
\end{proof}


\section{Modified Chromatic Schultz Polynomials}

\begin{definition}{\rm 
Let $G$ be a connected graph with chromatic number $\chi(G)$. Then, the \textit{modified chromatic Schultz polynomial} of $G$, denoted by $\sS^*(G,x)$, is defined as $$\sS^*(G,x)=\sum\limits_{u,v\in V(G)}(\zeta(u)\zeta(v))x^{d(u,v)}.$$  
}\end{definition}

\begin{definition}{\rm 
Let $G$ be a connected graph with chromatic number $\varphi^-$ and $varphi^+$ be the minimal and maximal parameter colouring of $G$. Then, 
\begin{enumerate}\itemsep0mm
\item[(i)] the \textit{modified $\chi^-$-chromatic Schultz polynomial} of $G$, denoted by $\cS^*(G,x)$, is defined as $$\cS^*(G,x)=\sum\limits_{u,v\in V(G)}(\zeta_{\varphi^-}(u)\cdot\zeta_{\varphi^-}(v))x^{d(u,v)};$$ and 
\item[(ii)] the \textit{$\chi^+$-chromatic Schultz polynomial} of $G$, denoted by $\dS^*(G,x)$, is defined as $$\dS^*(G,x)=\sum\limits_{u,v\in V(G)}(\zeta_{\varphi^+}(u)\cdot\zeta_{\varphi^+}(v))x^{d(u,v)}.$$
\end{enumerate} 
}\end{definition}

The following theorems discuss the modified chromatic Schultz polynomials of paths.

\begin{theorem}\label{Thm-ModPn1}
Let $P_n$ be a path on $n$ vertices. Then, we have 
$$\cS^*(P_n,x)=
\begin{cases}
\sum\limits_{i=0}^{\frac{n-1}{2}}\left[(2n-4i-2)x+(\frac{5n-10i-3}{2})\right]x^{2i}; & \text{if $n$ is odd};\\
\sum\limits_{i=0}^{\frac{n-1}{2}}\left[(2n-4i-2)x+(\frac{5n-10i}{2})\right]x^{2i}; & \text{if $n$ is even}.
\end{cases}$$
\end{theorem}
\begin{proof}
If $n$ is odd, then from Table \ref{Pn-odd}, we have 
\begin{eqnarray*}
\cS^*(P_n,x) & = & \sum\limits_{r\ {\rm odd}}2(n-r)x^r+\sum\limits_{r\ {\rm even}}\left[\frac{n-r+1}{2}\cdot 1+\frac{n-r-1}{2}\cdot 4\right]x^r\\
& = & \sum\limits_{r\ {\rm odd}}(2n-2r)x^r+\sum\limits_{r\ {\rm even}}\left(\frac{5n-5r-3}{2}\right)x^r\\
& = & \sum\limits_{i=0}^{\frac{n-1}{2}}(2n-4i-2)x^{2i+1}+ \sum\limits_{i=0}^{\frac{n-1}{2}}\left(\frac{5n-10i-3}{2}\right)x^{2i}\\
& = & \sum\limits_{i=0}^{\frac{n-1}{2}}\left[(2n-4i-2)x+\left(\frac{5n-10i-3}{2}\right)\right]x^{2i}.
\end{eqnarray*}

If $n$ is even, then from Table \ref{Pn-even}, we have 
\begin{eqnarray*}
\cS^*(P_n,x) & = & \sum\limits_{r\ {\rm odd}}2(n-r)x^r+\sum\limits_{r\ {\rm even}}\left[\frac{n-r}{2}\cdot 1+\frac{n-r}{2}\cdot 4\right]x^r\\
& = & \sum\limits_{r\ {\rm odd}}(2n-2r)x^r+\sum\limits_{r\ {\rm even}}\left(\frac{5n-5r}{2}\right)x^r\\
& = & \sum\limits_{i=0}^{\frac{n-1}{2}}(2n-4i-2)x^{2i+1}+ \sum\limits_{i=0}^{\frac{n-1}{2}}\left(\frac{5n-10i}{2}\right)x^{2i}\\
& = & \sum\limits_{i=0}^{\frac{n-1}{2}}\left[(2n-4i-2)x+\left(\frac{5n-10i}{2}\right)\right]x^{2i}.
\end{eqnarray*}
This completes the proof.
\end{proof}

\noi Similarly, by interchanging the colours $c_1$ and $c_2$, we have the following result.

\begin{theorem}\label{Thm-ModPn1a}
Let $P_n$ be a path on $n$ vertices. Then,  
$$\dS^*(P_n,x)=
\begin{cases}
\sum\limits_{i=0}^{\frac{n-1}{2}}\left[(2n-4i-2)x+(\frac{5n-10i+3}{2})\right]x^{2i}; & \text{if $n$ is odd};\\
\sum\limits_{i=0}^{\frac{n-1}{2}}\left[(2n-4i-2)x+(\frac{5n-10i}{2})\right]x^{2i}; & \text{if $n$ is even}.
\end{cases}$$
\end{theorem}

The following theorems discuss the modified chromatic Schultz polynomials of cycles.

\begin{theorem}\label{Thm-modCn1}
Let $C_n$ be a path on $n$ vertices. Then, we have 
$$\cS^*(C_n,x)=
\begin{cases}
\sum\limits_{i=0}^{\frac{n}{2}}\left(2nx+\frac{5n}{2}\right)x^{2i}; & \text{if $n$ is even};\\
\frac{5n+17}{2}+\sum\limits_{i=1}^{\frac{n-1}{2}}\left[(2n+9i+10)x+\frac{5n-18i+13}{2}\right]x^{2i}; & \text{if $n$ is odd}.
\end{cases}$$
\end{theorem}
\begin{proof}
If $n$ is even and $r=d(u,v), u,v\in V(C_n)$, then from Table \ref{Cn-even}, we have
\begin{eqnarray*}
\cS^*(C_n,x) & = & \sum\limits_{r\ {\rm odd}}\ 2nx^r + \sum\limits_{r\ {\rm even}}\ \left(\frac{n}{2}\cdot 1+\frac{n}{2}\cdot 4\right)x^{r}\\
& = & \sum\limits_{r\ {\rm odd}}\ 2nx^r + \sum\limits_{r\ {\rm even}}\ \frac{5n}{2}x^{r}\\
& = & \sum\limits_{i=0}^{\frac{n}{2}}2nx^{2i+1}+ \sum\limits_{i=0}^{\frac{n}{2}}\frac{5n}{2}x^{2i}\\
& =& \sum\limits_{i=0}^{\frac{n}{2}}\left[2nx+\frac{5n}{2}\right]x^{2i}.
\end{eqnarray*}

Let $n$ be odd. Then, from Table \ref{Cn-odd}, 
 
\begin{eqnarray*}
\cS^*(C_n,x) & = & \sum\limits_{r=0}(\frac{n-1}{2}\cdot 1+\frac{n-1}{2}\cdot 4+ 9\cdot 1)+\\ & &  \sum\limits_{r>0\ {\rm and\  odd}}\ \left(\frac{r-1}{2}(1+4)+2(n-r-1)+(3+6)\cdot 1\right)x^r +\\ & &  \sum\limits_{r>0\ {\rm and\ even}}\ \left((1+4)\frac{n-r-1}{2}+2(r-1)+(3+6)\cdot 1\right)x^{r}\\
& = & \frac{5n+17}{2}+\sum\limits_{r\ {\rm odd}}\ \frac{4n+9r+11}{2}x^r + \sum\limits_{r\ {\rm even}}\ \frac{5n-9r+13}{2}x^{r}\\
& = & \frac{5n+17}{2}+\sum\limits_{i=1}^{\frac{n-1}{2}}\ \frac{4n+18i+20}{2}x^{2i+1} + \sum\limits_{i=1}^{\frac{n-1}{2}}\ \frac{5n-18i+13}{2}x^{2i}\\
& =& \frac{5n+17}{2}+\sum\limits_{i=1}^{\frac{n-1}{2}}\left[(2n+9i+10)x+\frac{5n-18i+13}{2}\right]x^{2i}.
\end{eqnarray*}
This completes the proof.
\end{proof}

\noi Similarly, interchanging $c_1$ and $c_2$ in even cycles and interchanging $c_1$ and $c_3$ in even cycles, we get 

\begin{theorem}\label{Thm-modCn1a}
Let $C_n$ be a path on $n$ vertices. Then, we have 
$$\dS^*(C_n,x)=
\begin{cases}
\frac{13n-11}{2}+\sum\limits_{i=1}^{\frac{n-1}{2}}\left[(6n+i-7)x+\frac{13n-2i-15}{2}\right]x^{2i}; & \text{if $n$ is odd};\\
\sum\limits_{i=0}^{\frac{n}{2}}\left(2nx+\frac{5n}{2}\right)x^{2i}; & \text{if $n$ is even}.
\end{cases}$$
\end{theorem}

The following result provides the Schultz polynomial of a complete bipartite graph $K_{a,b}$.

\begin{theorem}\label{Thm-modKab}
For a complete bipartite $K_{a,b}, a\ge b, a+b=n$, we have  $\cS^*(K_n,x)=(a+4b)+2abx+\left(\frac{a(a-1)}{2}+2b(b-1)\right)x^2$ and  $\dS(K_n,x)=(4a+b)+2abx+\left(2a(a-1)+\frac{b(b-1)}{2}\right)x^2$.
\end{theorem}
\begin{proof}
The proof similar to that of Theorem \ref{Thm-Kab}.
\end{proof}


\section{Conclusion}

In this article, we have introduced a particular type of polynomial, called chromatic Schultz polynomial of graphs, as an analogue of the Schultz polynomial of graphs and determined this polynomial for certain fundamental graphs. 

The study seems to be promising for further studies as the polynomial can be computed for many graph classes and classes of derived graphs. The chromatic Schultz polynomial can be determined for graph operations, graph products and graph powers. The study on Schultz polynomials with respect to different types of graph colourings also seem to be much promising. The concept can be extended to edge colourings and map colourings also. 

These polynomials have so many applications in various fields like Mathematical Chemistry, Distribution Theory, Optimisation Techniques etc. In Chemistry, some interesting studies using the above-mentioned concepts are possible if $c(v_i)$ (or $\zeta(v_i)$) assumes the values such as energy, valency, bond strength etc. Similar studies are possible in various other fields. 

All these facts highlight the wide scope for further research in this area.

\section*{Acknowledgement} 

The author of this article would like to dedicate this article to his co-author and motivator Dr Johan Kok, Director, Licensing Services, Tshwane Metro Police Department, City of Tshwane, South Africa, as a tribute to his untiring efforts in the field of Mathematics research.



\begin{thebibliography}{99}

\bibitem{BM1} J.A. Bondy, U.S.R. Murty, (2008). {\em Graph theory}, Springer, New York.

\bibitem{BLS} A. Brandst\"{a}dt, V.B. Le and J.P. Spinrad, (1999). {\em Graph classes: A survey}, SIAM, Philadelphia.

\bibitem{CZ1} G. Chartrand, P. Zhang, (2009). \textit{Chromatic graph theory}, CRC Press, Boca Raton, FL, 2009.

\bibitem{ND} N. Deo, (1974). {\em Graph theory with application to engineering and computer science}, Prentice Hall of India, New Delhi.

\bibitem{ET1} M. Eliasi, B. Taeri, (2008). Schultz polynomials of composite graphs, 
{\it Appl. Anal. Discrete Math.}, {\bf 2}, 285-296, doi:10.2298/AADM0802285E.

\bibitem{EYY} S.P. Eu, B.Y. Yang, Y.N. Yeh, (2006). Theoretical and computational developments generalized Wiener indices in hexagonal chains, {\it Int. J. Quantum
Chem.}, {\bf 106}(2), 426-435.

\bibitem{FH} F. Harary, (2001). {\em Graph theory}, Narosa Publications, New Delhi.

\bibitem{JT1} T.R. Jensen, B. Toft, (1995). {\em Graph colouring problems}, John Wiley \& Sons, New York.

\bibitem{MK1} M. Kubale, (2004). {\em Graph colourings}, American Math. Soc., Rhode Island.

\bibitem{KSM} J.Kok, N.K. Sudev, U. Mary, (2017). On chromatic Zagreb indices of certain graphs,  \emph{Discrete Math. Algorithm. Appl.}, \textbf{9}(1), 1-11, DOI:10.1142/S1793830917500148.

\bibitem{EWW} E.W. Weisstein, (2011). {\em CRC concise encyclopedia of mathematics}, CRC press, Boca Raton.

\bibitem {DBW} D.B. West, (2001). {\em Introduction to graph theory}, Pearson Education, Delhi.

\end{thebibliography}
\end{document}